\documentclass[12pt]{article}

\usepackage{ graphicx, amssymb, pdfpages, 
	amsmath,amsthm,color,mathtools,tikz}
\usetikzlibrary{arrows,calc,intersections}
\usepackage{hyperref,url}

\usepackage{tkz-euclide}

\newtheorem{lemma}{Lemma}
\newtheorem{theorem}{Theorem}

\theoremstyle{definition}

\newtheorem{remark}{Remark}[section]

\begin{document}
\newcommand{\eps}{{\varepsilon}}
\newcommand{\proofend}{$\Box$\bigskip}
\newcommand{\C}{{\mathbb C}}
\newcommand{\Q}{{\mathbb Q}}
\newcommand{\R}{{\mathbb R}}
\newcommand{\Z}{{\mathbb Z}}
\newcommand{\RP}{{\mathbb {RP}}}
\newcommand{\CP}{{\mathbb {CP}}}
\newcommand{\Tr}{\rm Tr}
\newcommand{\g}{\gamma}
\newcommand{\G}{\Gamma}
\newcommand{\e}{\varepsilon}
\newcommand{\kk}{\kappa}

\title{Remarks on Joachimsthal integral and Poritsky property}

\author{Maxim Arnold\footnote{
Department of Mathematical Sciences, 
University of Texas at Dallas,  
Richardson, TX 75080; 
maxim.arnold@utdallas.edu
}
\and
Serge Tabachnikov\footnote{
Department of Mathematics,
Penn State University, 
University Park, PA 16802;
tabachni@math.psu.edu}
}

\date{}
\maketitle

\section{Joachimsthal integral} \label{sec:int}

The billiard inside an ellipse has a linear in momentum integral, the Joachimsthal integral. 

Let the ellipse be given by $Ax\cdot x =1$, where $A$ is a self-adjoint linear map, and let the phase space of the billiard map consist of pairs $(x,u)$ where $x$ is a point on the ellipse, and $u$ is an inward unit vector with foot point $x$ along the billiard trajectory. Let $y$ be the next intersection point of the trajectory with the ellipse and $v$ be the reflected unit vector at $y$. Then
$$
Ax\cdot u = - Ay\cdot u = Ay\cdot v,
$$
that is, $Ax\cdot u$ is an integral. See, e.g.,  \cite{Tab} for general information about mathematical billiards. 

The vector $Ay$ is normal to the conic at point $y$, and thus the second equality just corresponds to the billiard reflection law in an arbitrary curve and any normal vector. It is the equality $Ax\cdot u = - Ay\cdot u$ that is specific to conics. 
Indeed, $u$ is collinear with $y-x$ and, replacing $u$ with $y-x$, we have
\begin{equation}
	\label{eq:star}
	(Ax+Ay)\cdot (y-x) = Ax\cdot y -Ax\cdot x + Ay\cdot y - Ay\cdot x =0
\end{equation}
since $Ax\cdot x = Ay\cdot y =1$ and $Ax\cdot y = x\cdot Ay$.

In the present note we show that the existence of an integral which is linear in momentum is characteristic to conics. Let $\g$ be a convex, not necessarily closed, curve. Assume that $\g$ admits a non-vanishing normal vector field $N$ such that for every line that intersects $\g$ at two points $x$ and $y$ one has $N(x)\cdot (y-x) = - N(y) \cdot (y-x)$. 

We show that conics are characterized by this property and extend this result to conics in the spherical and hyperbolic geometries. We also consider the multidimensional case and show that an analogous property characterizes ellipsoids.

\section{Planar billiards} \label{sec:plane}

Let $\g(x)$ be a germ of a smooth  convex plane curve. 

\begin{theorem} \label{thm:plane}
Assume that $\g$ admits a non-vanishing normal vector field $N$ such that for every points $x,y\in\g$, one has
$$
N(x) \cdot (y-x) = -N(y) \cdot (y-x).
$$
Then $\g$ is a germ of a conic.
\end{theorem}

\begin{proof}
Let $\g(t)$ be an affine parameterization such that $[\g',\g'']=1$, where the brackets denote the determinant. Then $\g'''=-k\g'$, and the function $k(t)$ is called the affine curvature. Conics, and only conics, have constant affine curvature. See, e.g., \cite{Gug} for the  basics of affine differential geometry. 

Turning the normal vector $N(t)$ by $90^{\circ}$ we obtain the tangent field in the form $f(t)\g'(t)$, where $f(t)$ is an unknown function. Hence we reformulate the condition of the theorem
as
\begin{equation} \label{eq:aff1}
[f(t-\eps)\g'(t-\eps)+f(t+\eps)\g'(t+\eps),\g(t+\eps)-\g(t-\eps)]=0
\end{equation}
for all sufficiently small $\eps$.  The left hand side of the formula (\ref{eq:aff1}) is odd in $\eps$, and the first non-trivial term is cubic. 
Equating this cubic term to zero, we find that $f'=0$ and so $f(t)$ has to be constant. We can assume that $f\equiv 1$. 

Equating the quintic term to zero, we see that $[\g',\g^{\rm (v)}]=0$. One has
$$
\g^{\rm (iv)} = - k'\g'-k\g'',\ \g^{\rm (v)} = (k^2-k'')\g'-2k'\g'',
$$
hence $[\g',\g^{\rm (v)}]=-2k'$. Therefore $k$ is constant, and $\g$ is a conic. 
\end{proof}

\begin{remark}
The relation between arc length and affine parameterization is as follows. If $s$ is the arc length parameter and $t$ is the affine one, then $ds/dt=\kappa^{-1/3}$, where $\kappa$ is the curvature of the curve.

The condition $k'=0$ on the affine curvature can be expressed as the third order differential equation on $\kappa(s)$:
$$
36 \kk^4\kk'+9\kk^2\kk'''-45\kk\kk'\kk''+40(\kk')^3=0,
$$
where prime now stands for the derivative with respect to the arc-length parameter. Thanks to the above theorem, this equation characterizes conics. 
\end{remark}

Let us return to equation (\ref{eq:aff1}). Let $\g(x)$ be the parameterization such that $\dot \g = f \g'$, where dot denotes $d/dx$ and prime denotes $d/dt$.
Then we have
\begin{equation} \label{eq:arint1}
[\dot\g(x-\eps)+\dot\g(x+\eps),\g(x+\eps)-\g(x-\eps)]=0
\end{equation}
for every sufficiently small $\eps$. 

Denote by $A(x,y)$ the area bounded by the curve $\g$ and its chord $(\g(x), \g(y))$.

\begin{lemma} \label{lm:area1}
Fix a constant $c$ and assume that $x$ and $y$ are constrained by $y-x=c$. Then $A(x,y)$ is constant.
\end{lemma}

\begin{proof}
One has
$$
\frac{\partial A}{\partial x} = \left[\g(y)-\g(x),\dot\g(x)\right],  \
\frac{\partial A}{\partial y} = \left[\g(y)-\g(x),\dot\g(y)\right],
$$
and formula (\ref{eq:arint1}) implies the result via the chain rule.
\end{proof}

Thus Theorem \ref{thm:plane} can be restated as follows: {\it if a convex curve $\g$ admits a parameterization $\g(x)$ such that for every sufficiently small constant $c$ the area cut off from $\g$ by the 1-parameter family of chords $(\g(x), \g(x+c))$ is constant, then $\g$ is a conic}.  

This constant area property relates  our original (inner) billiard problem with outer billiards.

Let $\Gamma$ be a smooth strictly convex closed curve oriented counterclockwise. The outer (a.k.a. dual) billiard about $\Gamma$ is defined as follows. Let $x$ be a point outside of $\Gamma$. Draw the oriented tangent line to $\Gamma$ from $x$ and reflect $x$  in the tangency point to obtain a new point $y$. The outer billiard map takes $x$ to $y$. See \cite{DT} or \cite{Tab} for a survey of outer billiards.

Let $\g$ be an invariant curve of the outer billiard map. One can reconstruct the outer billiard curve $\Gamma$ by the {\it area construction}: $\Gamma$ is the envelope of the chords of $\g$ that cut off a constant area from $\g$. That is, this envelope touches the chords at their midpoints. 

A convex curve that admits a parameterization $\g(x)$ such that  the area cut off from $\g$ by the 1-parameter family of chords $\g(x) \g(x+c)$ is constant for all sufficiently small values of $c$ is said to possess the {\it area Poritsky property}.  It is named after Hillel Poritsky \cite{Por}, who studied its dual version related to the {\it string construction} for inner billiards that reconstructs the billiard curve from its caustic).

The Poritsky property was recently thoroughly studied in \cite{Gl,GIT}. In particular, Lemma \ref{lm:area1} shows that our Theorem \ref{thm:plane} is equivalent to the affine case of Theorem 1.13 of \cite{Gl}, and it provides a different proof of this result. 

To finish this section, let us return to the equality \eqref{eq:star}. Consider the gravitational law of attraction in the plane where the force is inverse proportional to the distance. The homeoid density on an ellipse $Ax\cdot x=1$ is the image of the uniform density on a circle under the affine map that takes the circle to the ellipse. That is, this density is  the area between the ellipse and the infinitesimally close homothetic ellipse, and it equals $1/|Ax|$. 

Newton's ``no gravity in a cavity" theorem states that an ellipse with the homeoid density exerts zero attraction at any interior point $O$. 

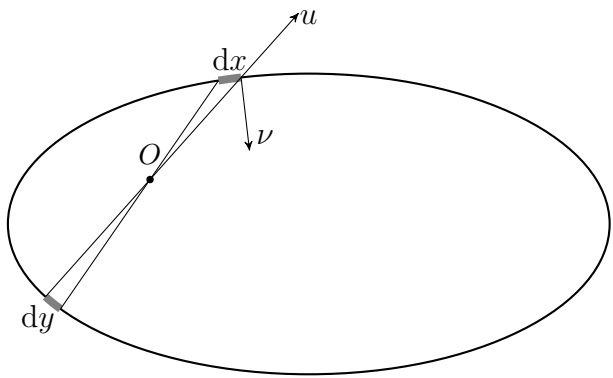
\begin{figure}[hbt]
	\centering
	\begin{tikzpicture}
		
			\tikzset{
			dot/.style={circle,inner sep=1pt,fill,name=#1,label={\small #1}}}
		
		\def\yi{-0.9};
		\def\yo{-3.5};
		
		\def\zi{-1.2};
		\def\zo{-3.3};
	 \draw [thick, name path=ell] (0,0) ellipse (4cm and 2cm); 
	 \draw [gray!100, line width = 3, domain=\yi:\zi] plot (\x, {sqrt(4-0.25*\x*\x)}); 
	 
	  \draw [gray!100, line width = 3, domain=\yo:\zo] plot (\x, -{sqrt(4-0.25*\x*\x)});

	 	\draw [name path=yiyo]($\yi*(1,0)+sqrt(16-\yi*\yi)*(0,0.5)$) -- ($\yo*(1,0)-sqrt(16-\yo*\yo)*(0,0.5)$); 
	 \draw[name path=zizo] ($\zi*(1,0)+sqrt(16-\zi*\zi)*(0,0.5)$) -- ($\zo*(1,0)-sqrt(16-\zo*\zo)*(0,0.5)$); 
	 \path [name intersections={of=yiyo and zizo, by={O}}];
	 \node[dot=$O$] (O) at (O){};
	  \draw[->,>=stealth'] ($\yi*(1,0)+sqrt(16-\yi*\yi)*(0,0.5)$)--($\yi*(1,0)+sqrt(16-\yi*\yi)*(0,0.5)-\yi*(0.125,0)-sqrt(16-\yi*\yi)*(0,0.25)$);
	 \draw[name path=circ, transparent]($\yi*(1,0)+sqrt(16-\yi*\yi)*(0,0.5)$) circle (1.15);
	 \path [name intersections={of=circ and yiyo, by={oo}}];
	 
	 \draw[->,>=stealth'] ($\yi*(1,0)+sqrt(16-\yi*\yi)*(0,0.5)$)--($\yi*(2,0)+sqrt(16-\yi*\yi)*(0,1)-(oo)$);
	 \node at ($\yi*(1,0)+sqrt(16-\yi*\yi)*(0,0.5)+(-0.15,0.25)$) {$\mathrm{d}x$}; 
	 \node at ($\yo*(1,0)-sqrt(16-\yo*\yo)*(0,0.5)-(0.1,0.3)$){$\mathrm{d} y$};
	 \node at ($\yi*(1,0)+sqrt(16-\yi*\yi)*(0,0.5)+(0.32,-0.8)$) {$\nu$}; 
	  \node at ($\yi*(1,0)+sqrt(16-\yi*\yi)*(0,0.5)+(0.9,0.8)$) {$u$}; 
 \end{tikzpicture}
\caption{Newton's ``no gravity in a cavity" theorem }
\end{figure}

Indeed, let $\ell$ be a line through $O$ intersecting the ellipse $\g$ at points $x$ and $y$. Turn $\ell$ through an infinitesimal angle $\eps$ about $O$, and let $dx$ and $dy$ be the  infinitesimal arcs of $\g$ cut off by the lines. Let $u$ be the unit vector from $O$ to $x$ and $\nu$ the unit normal vector to $\g$ at $x$. See Figure 1.

Then the arc length of $dx$ is $\eps |Ox|/ (u\cdot \nu)$, its mass is $\eps |Ox|/ ((u\cdot \nu) |Ax|)$, 
and the force exerted at $O$ is
$$
\frac{\eps}{|Ax|(u\cdot \nu)} = \frac{\eps}{Ax\cdot u}.
$$
A similar formula holds for the attraction of $dy$, and the formula $(Ax+Ay)\cdot u = 0$ means that these two forces cancel each other.

Therefore Theorem \ref{thm:plane} can be interpreted as saying that the only curves that admit the density for which the attraction forces locally cancel each other in this way are conics. See \cite{PP} for a different take on the same statement.

\section{Surfaces of constant curvature} \label{sec:sphere}
\subsubsection*{Spherical case}
Let $\mathbb{S}^2$ be the unit sphere.
A spherical conic is the intersection of $\mathbb{S}^2$ with a quadratic cone $Ax\cdot x = 0$ in $\R^3$. See \cite{Iz} for the geometry of spherical and hyperbolic conics.

First we verify that an analog of the Joachimsthal integral holds in spherical geometry.
Let $\g$ be a spherical conic, $x\in \g$ its point, $u$ an inward unit tangent vector at $x$. Let $y$ be the intersection point of the geodesic through $x$ in the direction of $u$ with $\g$, and let $v$ be the unit tangent vector to this geodesic at point $y$.

\begin{lemma} \label{lm:sphinv}
One has $Ax\cdot u = - Ay\cdot v$. 
\end{lemma}

\begin{proof}
Assuming that $x,y \in \mathbb{S}^2$ are distinct and non-antipodal points, we claim that 
\begin{equation} \label{eq:tang}
y-(x\cdot y)\ x\ \ {\rm and}\ \ (x\cdot y)\ y - x
\end{equation}
are oriented tangent vectors of the same length at points $x$ and $y$, respectively, to the oriented geodesic connecting $x$ and $y$.

Indeed, the first vector is orthogonal to $x$, and the second one to $y$, that is, they are tangent to the sphere at $x$ and $y$, respectively.  Both have length $\sqrt{1-(x\cdot y)^2}$, and both lie in the plane spanned by $x$ and $y$, hence they are tangent to the geodesic from $x$ to $y$. It remains to notice that they define the same orientation of the geodesic connecting $x$ and $y$.

Therefore we may replace $u$ and $v$ by $y-(x\cdot y)\ x$ and $(x\cdot y)\ y - x$, respectively. Then
$$
Ax \cdot (y-(x\cdot y)\ x) + Ay \cdot ((x\cdot y)\ y - x) = Ax\cdot y - Ay\cdot x =0
$$
since $Ax\cdot x = Ay\cdot y =0$. 
This completes the proof.
\end{proof}

The next theorem is a spherical analog of Theorem \ref{thm:plane}. 

\begin{theorem} \label{thm:sphere}
Let $\g$ be a smooth strictly convex spherical curve. Assume that $\g$ admits a non-vanishing normal  vector field $N$ (tangent to the sphere) with the following property:
for any points $x,y\in\g$, one has
$$
N(x)\cdot u =- N(y)\cdot v,
$$
where $u$ and $v$ are the unit tangent vectors at points $x$ and $y$ to the geodesic connecting $x$ and $y$. Then $\g$ is a (part of a) spherical conic.  
\end{theorem}

\begin{proof}
We argue similarly to the proof of Theorem \ref{thm:plane}.

Let us give the curve $\g$ an equiaffine parameterization such that $[\g,\g',\g'']=1$, where the brackets denote the $3\times 3$ determinant. Then $[\g,\g',\g''']=0$, hence 
\begin{equation} \label{eq:3rd}
\gamma'''=a \gamma+b \gamma',
\end{equation}
 where $a(t)$ and $b(t)$ are unknown functions.

As before, we   turn the normal vectors $90^\circ$ to make them tangent to $\g$ and, accordingly, replace dot product by cross-product, that is, the determinant of the position vector and the two tangent vectors involved. 

Write the tangent field as $f\g'$, where $f(t)$ is an unknown function. Let
$x=\g(t-\eps), y=\g(t+\eps)$ and, according to formula (\ref{eq:tang}), 
$$
u= \g(t+\eps) - (\g(t-\eps) \cdot \g(t+\eps)) \g(t-\eps),\ v = (\g(t-\eps) \cdot \g(t+\eps)) \g(t+\eps) - \g(t-\eps).
$$
We obtain a spherical analog of the equation (\ref{eq:aff1}):
\begin{equation} \label{eq:aff2}
f(t-\eps) [\g(t-\eps),\g'(t-\eps),\g(t+\eps)]- f(t+\eps) [\g(t+\eps),\g'(t+\eps),\g(t-\eps)]=0.
\end{equation}
As before, the left hand side is odd in $\eps$, and the first non-trivial term is cubic.

Evaluating this cubic term and using $[\g,\g',\g''']=0$, we find that $f$ is constant. Set $f\equiv 1$. 
Next, we evaluate the quintic term. Here is the calculation in which we shorthand $\gamma(t\pm \eps)$ as $(\gamma_{\pm})$.

One has
\[[\gamma_-,\gamma'_-,\gamma_+]-[\gamma_+,\gamma'_+,\gamma_-]=
(\gamma'_-\times \gamma_+-\gamma_+\times\gamma'_+)\cdot \gamma_-=\]
\[=(\gamma'_-+\gamma'_-)\times \gamma_+)\cdot \gamma_-=(\gamma_+\times \gamma_-)\cdot (\gamma'_-+\gamma'_+).\]
Expanding up to $\eps^5$, we get
\[\frac{1}{2} (\gamma_+\times \gamma_-)=\eps (\gamma'\times \gamma)+\eps^3(\frac{\gamma'''\times \gamma}{6}-\frac{\gamma''\times \gamma'}{2})+\eps^5(\frac{\gamma'''\times\gamma''}{12}-\frac{\gamma^{IV}\times \gamma'}{24}+\frac{\gamma^{V}\times \gamma}{120})\]
and 
\[\frac{1}{2} (\gamma'_-+\gamma'_+)=\gamma'+\eps^2\frac{\gamma'''}{2!}+\eps^4\frac{\gamma^V}{4!}.\]
 Therefore the quintic term is
 \[\frac{(\gamma^V\times \gamma)\cdot \gamma'}{30}+\frac{(\gamma'\times \gamma)\cdot \gamma^V}{6}-(\gamma''\times \gamma')\cdot \gamma'''+\frac{(\gamma'''\times \gamma'')\cdot \gamma'}{3}.\]
 
Differentiate  equation (\ref{eq:3rd}) twice to obtain
 \[\gamma^V=(a''+ab)\gamma+(2a'+b^2+b'')\gamma'+(a+2b')\gamma''\]
 and substitute in the formula above. This yields, up to a factor, the quintic term: $[\g,\g',\g''] (2a-b')$. It follows that $2a=b'$.

Finally, we use the characterization of projective conics in terms of the cubic differential equations (\ref{eq:3rd}). Namely, conics, and only conics, satisfy the relation $2a=b'$, see, e.g., section 1.4 of \cite{OT}. This concludes the proof.
\end{proof}

Next, we turn to the area Poritsky property. Consider equation (\ref{eq:aff2}) and let $\g(x)$ be a parameterization such that $\dot \g = f \g'$, where dot denotes $d/dx$ and prime denotes $d/dt$. Then
$$
[\g(x-\eps),\g(x+\eps),\dot\g(x-\eps)+\dot\g(x+\eps)]=0
$$
for all sufficiently small $\eps$. 

The above equation means that the velocity of the midpoint of the arc $\g(x-\eps) \g(x+\eps)$ is tangent to this arc as $x$ varies. As in the plane, this implies that the area bounded by the curve and this chord is constant. Thus $\g$ possesses the area Poritsky property, and our proof of Theorem \ref{thm:sphere} provides a different proof of the spherical  case of Theorem 1.13 of \cite{Gl}.

\subsubsection* {Hyperbolic case} \label{sec:hyp}

A version of  Theorem \ref{thm:sphere} holds in the hyperbolic plane as well. For this, consider the pseudosphere model of $\mathbb{H}^2$, that is, the upper sheet of the hyperboloid $x^2+y^2 -z^2 =-1$ in the Minkowski space with the metric $dx^2+dy^2-dz^2$.

Then the arguments of Section \ref{sec:sphere} apply with the appropriate changes of the signs in the formulas.
The area Poritsky property interpretation is valid as well, providing an alternative approach to  the hyperbolic  case of the Theorem 1.13 of \cite{Gl}.
We do not dwell on the details here.

We finish this section by two remarks. 

First, one has the spherical duality that interchanges points and great circles. Outer and inner billiards are dual to each other, therefore the area Poritsky property is dual to the usual Poritsky property related to the string construction; see \cite{Gl} for a detailed discussion of these matters.

Second, the gravitational interpretation extends to the spherical and hyperbolic geometries as well, see \cite{IT} for details.

\section{Higher dimensions} \label{sec:high}

 Let $S$ be a smooth closed strictly convex hypersurface in Euclidean space, the boundary of a billiard table. We have the following multi-dimensional analog of Theorem \ref{thm:plane}.
 
 \begin{theorem} \label{thm:high}
 Assume that $S$ admits a non-vanishing normal vector field $N$ such that for every points $x,y\in S$ one has
$$
N(x) \cdot (y-x) = -N(y) \cdot (y-x).
$$
Then $S$ is an ellipsoid.
 \end{theorem}

\begin{proof}
Let $\pi$ be a plane that transversally intersects $S$, and let $\g$ be the intersection curve. We claim that $\g$ is an ellipse.

To prove this, consider the orthogonal projection of the vectors $N$, taken at points of  $\g$, on the plane $\pi$. Denote this vector field along $\g$ by $\nu$. Since $\pi$ is transverse to $S$ and $N$ is orthogonal to it, the field $\nu$ is non-vanishing.

Let $x\in\g$ and let $\ell$ be the tangent line to $\g$ at $x$. Then $N(x) \perp \ell$, and $N(x)-\nu(x) \perp \pi$, hence $N(x)-\nu(x) \perp\ell$. Therefore $\nu = N - (N-\nu)$ is also an orthogonal vector field along $\g$. See Figure 2.  

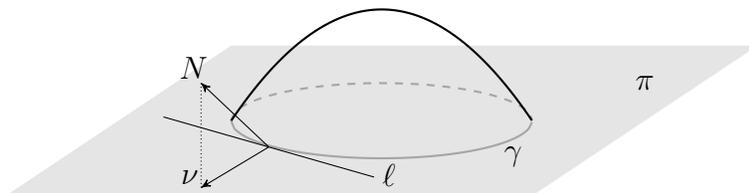
\begin{figure}[hbt]
	\centering
	\begin{tikzpicture}
		
		\tikzset{
			dot/.style={circle,inner sep=1pt,fill,name=#1,label={\small #1}}}
		\fill[gray!100, opacity=0.2] (2,0)--(5,2)--(12,2)--(9,0)--cycle;
\draw [gray!70,thick, dashed, domain=5:9, samples=100, smooth] plot (\x, {1+0.25*sqrt(4-(\x-7)*(\x-7))}); 
\draw [gray!70,thick, domain=5:9, samples=100, smooth] plot (\x, {1-0.25*sqrt(4-(\x-7)*(\x-7))}); 

\draw [thick, domain=5:9, samples=100, smooth] plot (\x, {1-0.37*(\x-5)*(\x-9)}); 

 \draw[->,>=stealth'] (5.5,0.65)--(4.6,0.11);
 \draw[densely dotted] (4.6,0.11)--(4.6,1.5);
 \draw (5.5,0.65)--(4.1,1.05);
 \draw (5.5,0.65)--(6.9,0.25);
 \draw[->,>=stealth'] (5.5,0.65)--(4.6,1.5);
 \node  at (10.5,1.5) {$\pi$};
  \node  at (4.5,1.7) {$N$}; 
  \node  at (7.1,0.28) {$\ell$};
   \node  at (4.45,0.25) {$\nu$};
  \node  at (8.75,0.53) {$\g$};
 
	\end{tikzpicture}
	\caption{Projection of the normal field to the plane.}
\end{figure}

Let $x,y\in \g$. We claim that $(\nu(x)+\nu(y))\cdot(y-x)=0$. Indeed, $\nu(x)=N(x) + (\nu(x)-N(x))$, and $(\nu(x)-N(x)) \cdot (y-x)$. Likewise, for $\nu(y)$. Therefore
$$
(\nu(x)+\nu(y))\cdot(y-x)=(N(x)+N(y))\cdot(y-x)=0.
$$
Now Theorem \ref{thm:plane} implies that $\g$ is an ellipse, as claimed.

Finally, according to  \cite[Lemma 12.1]{Gru}, if all 2-dimensional sections of $S$ are ellipses, then $S$ is an ellipsoid. This concludes the proof.
\end{proof}

\begin{remark}
The ``no gravity in a cavity" interpretation discussed at the end of Section \ref{sec:plane} applies in the multi-dimensional case as well: the gravitational attraction in $n$-dimensional space is  proportional to $r^{1-n}$.
\end{remark}

Theorem \ref{thm:high} also has a local version in which $S$ is not assumed to be a closed hypersurface. This follows from the next result that is of independent interest.

\begin{theorem} \label{thm:quadrics}
Let $S$ be a smooth hypersurface in the Euclidean space. Assume that every transverse 2-dimensional section of $S$ is a (part of a) conic. Then $S$ is a (part of a) quadric.
\end{theorem}

\begin{proof}
This proof was communicated to us by A. Glutsyuk; it is a simplified version of the argument of M. Berger \cite{Ber}, where a stronger statement is proved.

Let $x,y\in S$ be two points such that the line $xy$ is transverse to $S$ at $x$ and $y$. Let $Q$ be the quadric that shares the tangent hyperplanes with $S$ at points $x$ and $y$, and whose second quadratic form coincides with that of $S$ at point $x$. Below we will show that such a quadric exists.

We claim that $S=Q$. Indeed, consider a plane through the line $xy$. Its intersection with $S$ and $Q$ are conics, say, $C$ and $C'$. The local index of intersection of $C$ and $C'$ at $y$ is at least 2, and at $x$ it is at least 3. Hence the total index of intersection is at least 5, which implies that $C=C'$. This it true for all 2-planes containing $xy$, proving  the claim. 

It remains to construct the quadric $Q$. Applying a projective transformation, we may assume that $T_yS$ is the hyperplane at infinity, $T_xS$ is a ``horizontal" coordinate hyperplane, and the line $xy$ is the ``vertical" coordinate axis.
Applying an orthogonal transformation, we may assume that the second fundamental form is diagonal $\mathrm{diag}[a_1,\dots,a_n]$.  Then, in the Cartesian coordinates $(x_1,\dots,x_n,y)$, the hypersurface $S$ is a paraboloid given by the equation $y=\sum a_i x_i^2$.
\end{proof}

\bigskip

{\bf Acknowledgements}: We are grateful to A. Akopyan, M. Bialy,  A. Petrunin, and especially to A. Glutsyuk, for useful discussions. 
ST was supported by NSF grant DMS-2005444.


\begin{thebibliography}{99}

\bibitem{Ber} M. Berger. {\it Seules les quadriques admettent des caustiques.} Bull. Soc. Math. France {\bf 123} (1995), 107--116. 

\bibitem{DT} F. Dogru, S. Tabachnikov. {\it  Dual billiards.} Math. Intelligencer {\bf 27} (2005), no. 4, 18--25. 

\bibitem{Gl} A. Glutsyuk. {\it On curves with Poritsky property}. arXiv:1901.01881. 

\bibitem{GIT} A. Glutsyuk, I. Izmestiev, S. Tabachnikov. {\it Four equivalent properties of integrable billiards}. arXiv:1909.09028, Israel J. Math., to appear.

\bibitem{Gru} P. Gruber. {\it Convex and discrete geometry.} Springer, Berlin, 2007.

\bibitem{Gug} H. Guggenheimer. {\it Differential geometry.} McGraw-Hill, New York-San Francisco-Toronto-London, 1963.

\bibitem{Iz} I. Izmestiev. {\it Spherical and hyperbolic conics.} Eighteen essays in non-Euclidean geometry, 263--320, IRMA Lect. Math. Theor. Phys., {\bf 29}, Eur. Math. Soc., Zürich, 2019.

\bibitem{IT} I. Izmestiev, S.  Tabachnikov. {\it Ivory's theorem revisited.} J. Integrable Syst. {\bf 2} (2017), no. 1, xyx006, 36 pp.

\bibitem{OT} V. Ovsienko, S. Tabachnikov.
{\it Projective differential geometry old and new. From the Schwarzian derivative to the cohomology of diffeomorphism groups.}  Cambridge Univ. Press, Cambridge, 2005.

\bibitem{PP} A.  Panov, D.  Panov. {\it Homeoid density and Poncelet’s theorem}. Mat. Prosv.  ser. 3, v. {5}, MCCME, 
Moscow 2001, pp. 145--157. (Russian).

\bibitem{Por} H. Poritsky. {\it The billiard ball problem on a table with a convex boundary -- an illustrative dynamical problem}. Ann. of Math. {\bf 51} (1950), 446--470.

\bibitem{Tab} S. Tabachnikov. {\it Geometry and billiards}. American Math. Soc., Providence, RI, 2005.


\end{thebibliography}
\end{document}